\documentclass[preprint,12pt]{elsarticle}
\usepackage{amsmath,amssymb,amsthm,color,url}
\newtheorem{thr}{Theorem}

\newtheorem{claim}[thr]{Claim}
\newtheorem{obs}[thr]{Observation}

\theoremstyle{definition}

\theoremstyle{remark}

\numberwithin{equation}{section}
\journal{arXiv.org}

\begin{document}

\begin{frontmatter}

\title{Coloring the power graph of a semigroup}

\author{Yaroslav Shitov}

\ead{yaroslav-shitov@yandex.ru}

\address{National Research University Higher School of Economics, 20 Myasnitskaya Ulitsa, Moscow 101000, Russia}

\begin{abstract}
Let $G$ be a semigroup. The vertices of the power graph $\mathcal{P}(G)$ are the elements of $G$, and two elements are adjacent if and only if one of them is a power of the other. We show that the chromatic number of $\mathcal{P}(G)$ is at most countable, answering a recent question of Aalipour et al.
\end{abstract}


\end{frontmatter}

\section{Introduction}

This note is devoted to the graph constructed in a special way from a given semigroup $G$. This graph is called the \textit{power graph} of $G$, denoted by $\mathcal{P}(G)$, and its vertices are the elements of $G$. Elements $g,h\in G$ are adjacent if and only if one of them is a power of the other, that is, if we have either $g=h^k$ or $h=g^k$ for some $k\in\mathbb{N}$. (Here an in what follows, $\mathbb{N}$ denotes the set of positive integers.) This concept has attracted some attention in both the discrete mathematics and group theory, see~\cite{Cam, CG}.

Let us consider a related directed graph $D(G)$ on the set $G$. Assume $D(G)$ contains an edge leading from $x$ to $y$ if and only if $y$ is a power of $x$. Clearly, the outdegree of any vertex of $D(G)$ is at most countable, and one gets the graph $\mathcal{P}(G)$ by forgetting about the orientations of the edges of $D(G)$. A classical result by Fodor (see~\cite{Fod, Kom}) shows that the chromatic number of $\mathcal{P}(G)$ does not exceed any uncountable cardinal. Is this chromatic number always at most countable? This question was studied in~\cite{AAC} by Aalipour et al. They answer this question in the special cases which include groups with finite exponent, free groups, and Abelian groups. However, the general version of this problem remained open even for groups, and it was posed in~\cite{AAC} as Question~42. This note gives an affirmative answer to this question. More than that, we prove that the chromatic number of $\mathcal{P}(G)$ is countable even if $G$ is an arbitrary power-associative magma. (Here, a \textit{magma} is a set endowed with a binary operation. A magma is called \textit{power-associative} if a sub-magma generated by a single element is associative.)

We proceed with the proof. The \textit{order} of an element $h\in G$ is the cardinality of the subsemigroup generated by $h$. An element $h\in G$ is called \textit{cyclic} if the subsemigroup generated by $h$ is a finite group. In other words, an element $h$ is cyclic if and only if the equality $h=h^{n+1}$ holds for some positive integer $n$. If $h$ is not cyclic but has a finite order, then the \textit{pre-period} of $h$ is defined as the largest $p$ such that the element $h^p$ occurs in the sequence $h,h^2,h^3,\ldots$ exactly once.

\section{Coloring the elements of finite orders}

The following claim allows us to split the set of all cyclic elements into a union of countably many independent sets.

\begin{claim}\label{claim1}
Fix a number $n\in\mathbb{N}$. The subgraph of $\mathcal{P}(G)$ induced by the set of cyclic elements of order $n$ is a union of cliques of size at most $n$.
\end{claim}

\begin{proof}
Denote this induced subgraph by $P'$. Let $\thicksim$ be the relation on $P'$ containing those pairs $(x,y)$ such that $x$ is a power of $y$; this relation is clearly reflexive and transitive. Assuming that $x\thicksim y$, we get $x=y^p$, and we note that $p$ is relatively prime to $n$ because the orders of $x,y$ are equal to $n$. So we get $pq+p'n=1$ for some $p'\in\mathbb{Z}$, $q\in\mathbb{N}$, which shows that $y=x^q$. Therefore, $\thicksim$ is an equivalence relation, and every equivalence class is a subset of the set of powers of some $x\in P'$. 
\end{proof}

As we see from the proof, the sizes of the cliques as in Claim~\ref{claim1} are equal to $\varphi(n)$, where $\varphi$ is the Euler's totient function. This result is similar to Theorem~15 in~\cite{AAC}. Now we are going to prove that the set of all non-cyclic elements of finite orders can be represented as a union of countably many independent sets. We need the following observation.

\begin{obs}\label{obs111}
If $g\in G$ has a finite order $n$ and pre-period $p$, then $g^q$ is cyclic for all $q>p$.
\end{obs}

\begin{proof}
We have $g^{n+1}=g^{p+1}$, so that $(g^q)^{n-p+1}=g^{p+1}g^{(n-p)q}g^{q-p-1}=g^q$.
\end{proof}

\begin{claim}\label{claim11}
Let $g,h\in G$ be distinct elements with finite orders. If $g,h$ have the same pre-period $p$, then they are non-adjacent in $\mathcal{P}(G)$. 
\end{claim}

\begin{proof}
Assume the result is not true. Then we have $g=h^t$, for some $t>1$. (We omit the case $h=g^t$, which is considered similarly.) Observation~\ref{obs111} shows that $g^p=h^{tp}$ is a cyclic element, which contradicts to the initial assumption that $p$ is the pre-period of $g$.
\end{proof}

\section{Coloring the elements of the infinite order}

In the following claim, we assume $m,n\in\mathbb{N}$, and we denote by $G(x,m,n)$ the set of all $y\in G$ satisfying $x^m=y^n$.

\begin{claim}\label{claim2}
Let $x\in G$ be an element of the infinite order. Then the set $G(x,m,n)$ is independent in $\mathcal{P}(G)$.
\end{claim}

\begin{proof}
Assume that $k\in\mathbb{N}$ and $y,z\in G$ are such that $x^m=y^n$, $x^m=z^n$, $y=z^k$. Then we have $x^{mk}=z^{kn}=y^n=x^m$. Since $x$ has the infinite order, we get $k=1$, which implies $y=z$ and completes the proof.
\end{proof}

In what follows, we denote by $\pi\subset G$ the set of elements of finite orders and by $\mathcal{P}_*(G)$ the graph obtained from $\mathcal{P}(G)$ by removing the vertices in $\pi$.

\begin{claim}\label{claim3}
Let $x\in G$ be an element of the infinite order. We define the set $C(x)=\bigcup_{m,n\in\mathbb{N}}G(x,m,n)$. Then $C(x)$ is a connected component of $\mathcal{P}_*(G)$. 
\end{claim}

\begin{proof}
If we have $x^{m_1}=g^{n_1}$, $x^{m_2}=h^{n_2}$ with $m_1,m_2,n_1,n_2\in\mathbb{N}$, then both $g$ and $h$ are adjacent to $x^{m_1m_2}\in C(x)$, which shows that $C(x)$ is connected. Now assume that an element $z\in \mathcal{P}_*(G)$ is adjacent to a vertex $y$ in $C(x)$. Then we have $x^m=y^n$ and $y^{p}=z^{q}$ with positive integers $m,n,p,q$ (and either $p=1$ or $q=1$, but this fact is not relevant for our proof). We get $x^{mp}=z^{nq}$, which implies that $z$ belongs to $C(x)$ as well. In other words, the vertices in $C(x)$ can be adjacent only to vertices in $C(x)$.
\end{proof}

Now we are ready to prove our main result, which states that $G$ is a union of countably many independent subsets of $\mathcal{P}(G)$. Claim~\ref{claim1} shows that the subgraph of $\mathcal{P}(G)$ induced by the cyclic elements can be covered by countably many independent sets. Claim~\ref{claim11} proves the same result for the subgraph induced by those elements that have finite orders but are not cyclic. These claims together allow us to cover the set $\pi$ by countably many independent sets of $\mathcal{P}(G)$.

We denote by $\{C_\alpha\}$ the set of all connected components of the graph $\mathcal{P}_*(G)$, which is obtained from $\mathcal{P}(G)$ by removing the vertices in $\pi$. We choose an element $x_\alpha$ in every connected component $C_\alpha$, and we deduce from Claim~\ref{claim3} that $C_\alpha=C(x_\alpha)$ for all indexes $\alpha$. Claim~\ref{claim2} shows that every $C(x_\alpha)$ is the union of the independent sets $G(x_\alpha,m,n)$ over all pairs of positive integers $(m,n)$. We see that $G\setminus\pi$ is the union of the independent sets $\cup_{\alpha}G(x_\alpha,m,n)$, which completes the proof.


\begin{thebibliography}{99}

\bibitem{AAC}
G. Aalipour, S. Akbari, P. J. Cameron, R. Nikandish, F. Shaveisi, On the structure of the power graph and the enhanced power graph of a group, preprint (2016) arXiv:1603.04337.

\bibitem{Cam}
P. J. Cameron, The power graph of a finite group, II, J. Group Theory 13 (2010) 779--783.

\bibitem{CG}
P. J. Cameron, S. Ghosh, The power graph of a finite group, Discrete Math. 311 (2011) 1220--1222.

\bibitem{Fod}
G. Fodor, Proof of a conjecture of P. Erd\H{o}s, Acta. Sci. Math. Hung. 14 (1952) 219--227.

\bibitem{Kom}
P. Komj\'{a}th,  A note on uncountable chordal graphs, Discrete Math. 338 (2015) 1565--1566.
\end{thebibliography}
\end{document}